\documentclass[10pt,a4paper]{amsart}

\title[Semigroups generated by elliptic operators in non-divergence form ]
{Semigroups generated by elliptic operators in non-divergence form on
$C_0(\Omega)$}

\author{Wolfgang Arendt }
\address{Institute of Applied Analysis,  University of Ulm, D - 89069 Ulm, 
Germany}
\email{wolfgang.arendt@uni-ulm.de}

\author{Reiner Sch\"atzle} 
\address{Institute of Mathematics, Eberhard-Karls-University of T\"ubingen, 
D-72076 T\"ubingen, Germany}
\email{rscha@everest.mathematik.uni-tuebingen.de} 

\subjclass{35K20, 35J25, 47D06}

\keywords{holomorphic semigroups, elliptic operators in non-divergence form,
Dirichlet problem, Wiener regular, Lipschitz domain, exterior cone property}

\date{September 30, 2010}

\newtheorem{theorem}{Theorem}[section]
\newtheorem{lemma}[theorem]{Lemma}
\newtheorem{proposition}[theorem]{Proposition}
\newtheorem{corollary}[theorem]{Corollary}

\theoremstyle{definition}

\theoremstyle{remark}

\numberwithin{equation}{section}

\newcommand{\R}{{\mathbb R}}

\newcommand{\N}{{\mathbb N}}

\newcommand{\C}{{\mathbb C}}

\newcommand{\cA}{{\mathcal A}}
\newcommand{\cB}{{\mathcal B}}
\newcommand{\cD}{{\mathcal D}}

\newcommand{\cL}{{\mathcal L}}

\DeclareMathOperator{\diam}{diam}
\DeclareMathOperator{\dist}{dist}
\DeclareMathOperator{\loc}{loc}
\DeclareMathOperator{\supp}{supp}

\def\Re{{\rm Re \,}}

\newcommand{\dy}{{\,dy}}

\begin{document}

\maketitle

\begin{abstract}\label{abstract}
Let $\Omega \subset \R^n$ be a bounded open set satisfying the uniform exterior
cone condition. Let $\cA$ be a uniformly elliptic operator given by
$$
\cA u= 
\sum\limits^n_{i,j=1}a_{ij}\partial_{ij} u + \sum\limits^n_{j=1}b_j \partial_j u
+ cu
$$
where 
$$
a_{ij}\in C(\bar{\Omega}) \ \mbox{ and } \ b_j, c \in L^\infty (\Omega), c \le 
0 \ .
$$
We show that the realization $A_0$ of $\cA$ in
$$
C_0(\Omega):=\{u\in C(\bar{\Omega}):u_{|_{\partial\Omega}}=0\}
$$
given by
\begin{eqnarray*}
D(A_0)&:=&\{u\in C_0(\Omega)\cap W^{2,n}_{\loc}(\Omega):\cA u \in
C_0(\Omega)\}\\
A_0u&:=& \cA u
\end{eqnarray*}
generates a bounded holomorphic $C_0$-semigroup on $C_0(\Omega)$. The result is
in particular true if $\Omega$ is a Lipschitz domain. So far the best known
result
seems to be the case where $\Omega$ has $C^2$-boundary \cite[Section
3.1.5]{Lu}. We also study the elliptic problem
$$
\begin{array}{l}
-\cA u \ = \ f\\
u_{|_{\partial\Omega}} \ =\ g \ .
\end{array}
$$
\end{abstract}

\setcounter{section}{-1}
\section{Introduction} \label{intro}
The aim of this paper is to study elliptic and parabolic problems for operators
in non-divergence form with continuous second order coefficients and to prove
the existence (and uniqueness) of solutions which are continuous up to the
boundary of the domain. Throughout this paper $\Omega$  is a bounded open set
in $\R^n,n\ge 2$, with boundary $\partial \Omega$. We consider the operator
$\cA$ given by
$$
\cA
u:=\sum\limits^n_{i,j=1}a_{ij}\partial_{ij}u+\sum\limits^n_{j=1}b_j\partial_j
u+cu
$$
with real-valued coefficients $a_{ij},b_j,c$ satisfying
$$
\begin{array}{l}
b_j\in L^\infty(\Omega) \ , \ j=1,\ldots,n \ , \ c\in L^\infty(\Omega) \ , \ c
\le 0
\,\\
a_{ij}\in C(\bar{\Omega}) \ , \  a_{ij}=a_{ji}   \ ,\\
\sum\limits^n_{i,j=1}a_{ij}(x)\xi_i\xi_j \ge \Lambda |\xi|^2 \qquad (x \in
\bar{\Omega},\xi \in \R^n)
\end{array}
$$
where $\Lambda > 0$ is a fixed constant.\\
Our best results are obtained under the hypothesis that $\Omega$ satisfies the
uniform exterior cone condition (and thus in particular if $\Omega$ has
Lipschitz boundary). Then we show that for each $f\in L^n(\Omega),g\in
C(\partial \Omega)$ there exists a unique $u\in C(\bar{\Omega}) \cap
W^{2,n}_{\loc}(\Omega)$  such that
$$
 (E)
\left\{  \begin{array}{lll}
-\cA  u&=& f\\
u_{|_{\partial\Omega}}&=& g \ .
\end{array} \right.
$$
(Corollary \ref{cor2.3}). This result is proved with the help of Alexandrov's
maximum principle (which is responsible for the choice of $p=n$)  and other
standard results for elliptic second order differential operators (put together
in the appendix). Our main concern is the parabolic problem
$$
 (P)
\left\{  \begin{array}{rll} 
u_t &=&\cA u\\
u(0,\cdot)&=&u_0\\
u(t,x)&=&0 \qquad x \in \partial \Omega \ , \ t> 0 \ .
         \end{array}\right.
$$
with Dirichlet boundary conditions. Let $C_0(\Omega):=\{v\in
C(\bar{\Omega}):v_{|_{\partial \Omega}}=0\}$.
Under the uniform exterior cone condition, we show that the realization $A_0$
of $\cA$ in $C_0(\Omega)$ given by
\begin{eqnarray*}
D(A_0)&:=&\{v\in C_0(\Omega)\cap W^{2,n}_{\loc}(\Omega):\cA v \in
C_0(\Omega)\}\\
A_0 v&:=&\cA v
\end{eqnarray*}
generates a bounded, holomorphic $C_0$-semigroup on $C_0(\Omega)$. This 
improves the known results, which are presented in the monographie of Lunardi
\cite[Corollary 3.1.21]{Lu} for $\Omega$ of class $C^2$ (and $b_j,c$ uniformly 
continuous).

If the second order coefficients are Lipschitz continuous, then the results 
mentioned  so far hold if $\Omega$ is merely Wiener-regular. For elliptic
operators   in divergence form, this is proved in \cite[Theorem 8.31]{gil.tru}
for the elliptic problem
$(E)$ and in \cite[Corollary 4.7]{AB99} for the parabolic problem $(P)$.
Concerning the
elliptic problem $(E)$, and in particular the Dirichlet problem; i.e., the case
$f=0$ in $(E)$, there is earlier work by Krylov \cite[Theorem 4]{Kry67}, who
shows well-posedness of the Dirichlet problem if $\Omega$ is merely Wiener
regular and the second order coefficients are Dini-continuous. Krylov also
obtains the well-posedness of the Dirichlet problem for $a_{ij}\in C
(\bar{\Omega})$ if $\Omega$ satisfies the uniform exterior cone condition
\cite[Theorem 5]{Kry67}. He uses different (partially probabilistic) methods,
though.

\section{The Poisson problem}\label{secPoisson}
We consider the bounded open set $\Omega \subset \R^n$ and the elliptic 
operator $\cA$ from the Introduction. At first we consider the case where the
second order conditions are Lipschitz continuous. Then we merely need a very
mild regularity condition on $\Omega$. We say that $\Omega$ is \textit{Wiener
regular} (or \textit{Dirichlet regular}) if for each $g\in C(\partial \Omega)$
there exists a solution $u\in C^2(\Omega)\cap C(\bar{\Omega})$ of the Dirichlet
problem
\begin{eqnarray*}
\Delta u&=& 0\\
u_{|_{\partial\Omega}}&=&g \ .
\end{eqnarray*}
If $\Omega$ satisfies the exterior cone condition, then $\Omega$ is Dirichlet
regular.

\begin{theorem}\label{thm1.1}
Assume that the second order coefficients $a_{ij}$ are globally Lipschitz
continuous. If $\Omega$ is Wiener-regular, then for each $f \in L^n(\Omega)$,
there exists a unique $u\in W^{2,n}_{\loc}(\Omega)\cap C_0(\Omega)$ such that
$$
-\cA u = f \ .
$$
\end{theorem}

The point is that for Lipschitz continuous $a_{ij}$ the  operator $\cA$ may be
written in divergence form. This is due to the following lemma.

\begin{lemma}\label{lem1.2}
Let $h:\Omega \to \R$ be Lipschitz continuous. Then $h\in
W^{1,\infty}(\Omega)$. In particular, $hu \in W^{1,2}(\Omega)$ for all $u \in
W^{1,2}(\Omega)$ and $\partial_j (hu)=(\partial_j h)u+h\partial_j u$.
\end{lemma}

\begin{proof}
One can extend $h$ to a Lipschitz function on $\R^n$ (without increasing the
Lipschitz constant, see \cite{Min}). Now  the result follows from
\cite[5.8 Theorem 4]{Eva98}.
\end{proof}

\noindent
\textbf{Proof of Theorem \ref{thm1.1}.}
We assume that $\Omega$ is Dirichlet regular. Uniqueness follows from
Aleksandrov's maximum principle Theorem \ref{thmA1}.
In order to solve the problem we replace $\cA$ by an operator in divergence
form in the following way.
Let $\tilde{b}_j:=b_j-\sum\limits^n_{i=1}\partial_ia_{ij},j=1,\ldots,n$. Then
$\tilde{b}_j \in L^\infty(\Omega)$. Consider the elliptic operator $\cA_d$ in 
divergence form given by 
$$
\cA_d u = \sum\limits^n_{i,j=1}\partial_i(a_{ij}\partial_j
u)+\sum\limits^n_{j=1} \tilde{b}_j \partial_j u+cu \ .
$$
a) Let $f\in L^q(\Omega)$  for $q>n$. By \cite[Theorem 8.31]{gil.tru} or
\cite[Corollary 4.6]{AB99} there exists a unique $u \in C_0(\Omega)\cap
W^{1,2}_{\loc}(\Omega)$ such that $-\cA_d   u = f$ weakly, i.e.,
$$
\sum\limits^d_{i,j=1}\int\limits_\Omega a_{ij}\partial_j u \partial_i v -
\sum\limits^d_{j=1}\int\limits_\Omega \tilde{b}_j\partial_j u v -
\int\limits_\Omega cuv = \int\limits_\Omega f v
$$
for all $v\in \cD(\Omega)$ (the space of all test functions). We mention
in passing that $u \in W^{1,2}_0(\Omega)$ by \cite[Lemma 4.2]{AB99}.
For our purposes, it is important that $u \in W^{2,2}_{\loc}(\Omega)$ by
Friedrich's theorem \cite[Theorem 8.8]{gil.tru}. Here we use again that the
$a_{ij}$ are uniformly Lipschitz continuous but do not need any further
hypothesis on $b_j$ and $c$. It follows from Lemma \ref{lem1.2} that
$a_{ij}\partial_ju \in W^{1,2}_{\loc}(\Omega)$ and
$\partial_i(a_{ij}\partial_ju)=(\partial_i
a_{ij})\partial_ju+a_{ij}\partial_{ij} u$.
Thus $\cA_d u = \cA u$. Now  it follows from the interior Calderon-Zygmund
estimate Theorem \ref{thmA2} that $u \in W^{2,q}_{\loc}(\Omega)\subset
W^{2,n}_{\loc}(\Omega)$. This settles the result if $f \in L^q(\Omega)$ for
some $q>n$.\\
b) Let $f\in L^n(\Omega)$. Choose $f_k\in L^\infty(\Omega)$ sucht that
$\lim\limits_{k\to \infty}f_k=f$ in $L^n(\Omega)$. Let $u_k\in
W^{2,n}_{\loc}\cap C_0(\Omega)$ such that $-\cA u_k=f_k$ (use case a)). By
Aleksandrov's maximum principle Thereom \ref{thmA1}, we have
$$
\|u_k-u_\ell\|_{L^{\infty}(\Omega)}\le c\|f_k-f_\ell\|_{{L^{n}(\Omega)}} \ .
$$
Thus $u_k$  converge uniformly to a function $u \in C_0(\Omega)$ as $k\to
\infty$. By the Calderon-Zygmund estimate (Theorem \ref{thmA2}),
$$
\|u_k\|_{W^{2,n}(B_{\varrho})} \le
c(\|u_k\|_{L^{n}(B_{2\varrho})}+\|f_k\|_{L^{n}(B_{2\varrho})})
$$
if $\overline{B_{2\varrho}}\subset \Omega$, where the constant $c$ does
not depend
on $k$. Thus the sequence $(u_k)_{k\in\N}$ is bounded in $W^{2,n}(B_\varrho)$.
It follows from reflexivity that $u\in W^{2,n}(B_\varrho)$ and
$u_k\rightharpoonup u$ in $W^{2,n}(B_\varrho)$
as $k\to \infty$ after extraction of a subsequence. Consequently, $u \in
W^{2,n}_{\loc}(\Omega)\cap C_0(\Omega)$. Since $-\cA u_k=f_k$ for all $k \in
\N$, it follows that $-\cA u=f$.
\qed\\

Now we  return to the general assumption  $a_{ij}\in C(\bar{\Omega})$ and do no
longer assume that the $a_{ij}$ are Lipschitz continuous. We need the
following lemma which we prove for convenience.

\begin{lemma}\label{lem1.3}
a) There exist $\tilde{a}_{ij}\in C^b(\R^n)$ such that
$\tilde{a}_{ij}=\tilde{a}_{ji},\tilde{a}_{ij}(x)=a_{ij}(x)$ if $x\in
\Omega$ and
$$
\sum\limits^n_{i,j=1}a_{ij}(x)\xi_i\xi_j \ge \frac{\Lambda}{2}|\xi|^2
$$
for all $\xi \in \R^n,x\in \Omega$.\\
b) There exist $a^k_{ij}\in C^\infty(\bar{\Omega})$ such that
$a^k_{ij}=a^k_{ji},\sum\limits^n_{i,j=1}a^k_{ij}(x)\xi_i\xi_j
\ge \frac{\Lambda}{2}|\xi|^2$ and
$\lim\limits_{k\to\infty}a^k_{ij}(x)=a_{ij}(x)$  uniformly on $\bar{\Omega}$.
\end{lemma}

\begin{proof}
a) Let $b_{ij}:\R^n\to \R$ be a bounded, continuous extension of $a_{ij}$ to
$\R^n$. Replacing $b_{ij}$ by $\frac{b_{ij}+b_{ji}}{2}$, we may assume that
$b_{ij}=b_{ji}$. Since the function $\varphi :\R^n \times S^1 \to \R$ given by
$\varphi(x,\xi):=\sum\limits^n_{i,j=1}b_{ij}(x)\xi_i\xi_j$ is continuous and
$S^1:=\{\xi \in \R^n:|\xi|=1\}$ is compact, the set $\Omega_1:=\{x\in
\R^n:\varphi(x,\xi)>\frac{\Lambda}{2}$ for all $\xi \in S^1\}$ is open and
contains $\bar{\Omega}$. Let $0\le \varphi_1,\varphi_2 \in C(\R^n)$ such that
$\varphi_1(x)+\varphi_2(x)=1$ for all $x \in \R^n$  and $\varphi_2(x)=1$ for
$x\in\R^n\setminus\Omega_1,\varphi_1(x)=1$ for $x \in \bar{\Omega}$. Then
$\tilde{a}_{ij}:=\varphi_1 b_{ij}+\frac{\Lambda}{2}\varphi_2 \delta_{ij}$
fulfills the requirements.\\
b) Let $(\varrho_k)_{k\in\N}$ be a mollifier satisfying $\supp \varrho_k \subset
B_{1/k}(0)$. Then $a^k_{ij}=\tilde{a}_{ij}\ast \varrho_k\in C^\infty (\R^n)$
and $\lim\limits_{k\to \infty} a^k_{ij}(x)=\tilde{a}_{ij}(x)=a_{ij}(x)$
uniformly in $x \in \bar{\Omega}$. If $\frac{1}{k}<\dist
(\partial\Omega_1,\Omega)$, then for $x \in \Omega,\xi \in\R^n$
\begin{eqnarray*}
\sum\limits^n_{i,j=1}a^k_{ij}(x)\xi_i\xi_j&=&\int\limits_{|y|<1/k}\sum\limits^n_
{i,j=1}\tilde{a}_{ij}(x-y)\xi_i\xi_j\varrho_k(y)\dy\\
&\ge&\frac{\Lambda}{2}\int\limits_{|y|<1/k}\varrho_k(y)\dy=\frac{\Lambda}{2} \ .
\end{eqnarray*}
\end{proof}

\begin{theorem}\label{thm1.4}
Assume that $\Omega$ satisfies the uniform exterior cone condition. Then for
all $f\in L^n(\Omega)$ there exists a unique $u\in C_0(\Omega)\cap
W^{2,n}_{\loc}(\Omega)$ such that $-\cA u = f$.
\end{theorem}

\begin{proof}
As for Theorem \ref{thm1.1} we merely  have to prove existence of a solution.
We choose $a^k_{ij}\in C^\infty(\bar{\Omega})$ as im Lemma
\ref{lem1.3}.
Let $\cA_k$ be the elliptic operator with the second order coefficients
$a_{ij}$ of $\cA$ replaced by $a^k_{ij}$. Let $f \in L^n(\Omega)$. By
Theorem \ref{thm1.1}, for each $k\in \N$ there exists a unique $u_k\in
W^{2,n}_{\loc}(\Omega)\cap C_0(\Omega)$ such that $-\cA_k u_k=f$.
By H\"older regularity (Theorem \ref{thmA3}) there exists a constant $c$ which
does not depend on $k\in\N$ such that
$$
\|u_k\|_{C^\alpha(\Omega)}\le c (\|f\|_{L^n(\Omega)}+\|u_k\|_{L^n(\Omega)}) \ .
$$
By Aleksandrov's maximum principle $\|u_k\|_{L^\infty(\Omega)}\le
2c_1\|f\|_{L^n(\Omega)}$ for all $k\in\N$ and some constant $c_1$. Notice that
the first order coefficients of $\cA_k$ are independent of $k \in \N$. Thus
$(u_k)_{k\in\N}$ is bounded in $C^\alpha(\Omega)$. By the Arcela-Ascoli theorem
we may assume that $u_k$ converges uniformly to $u\in C_0(\Omega)$ as $k\to
\infty$ (passing to a subsequence of necessary). Let
$\overline{B_{2\varrho}}\subset \Omega$ where $B_{2\varrho}$ is a ball of radius
$2\varrho$. Since the
modulus of continuity of the $a^k_{ij}$ is bounded, by the interior
Calderon-Zygmund estimate Theorem \ref{thmA2}
$$
\|u_k\|_{W^{2,n}(B_\varrho)}\le
c_2(\|u_k\|_{L^n(B_{2\varrho})}+\|f\|_{L^n(B_{2\varrho})})
$$
for all $k\in \N$ and  some constant $c_2$. It follows from reflexivity that $u
\in W^{2,n}(B_\varrho)$ and $u_k\rightharpoonup u$ in $W^{2,n}(B_\varrho)$ as
$k\to \infty$ after extraction of a subsequence.
Since $-\cA_k u_k=f$, it follows that $-\cA u=f$. In fact, since
$u_k\rightharpoonup u$
weakly in $W^{2,n}(B_\varrho)$, it follows that $\partial_{ij}u_k\rightharpoonup
\partial_{ij}u$ in $L^n(B_\varrho)$ as $k\to \infty$. Thus
$\sup\limits_k\|\partial_{ij}u_k\|_{L^n(B_\varrho)}<\infty$. It follows that
$$
(a^k_{ij}-a_{ij})\partial_{ij}u_k\to 0 \ \mbox{ in } \ L^n(B_\varrho) \ \mbox{
as } \ k \to \infty
$$
and consequently $a^k_{ij}\partial_{ij}u_k\rightharpoonup a_{ij}\partial_{ij}u$
in
$L^n(B_\varrho)$.
\end{proof}

\section{The Dirichlet problem} \label{secDirichlet}
In this section we show the equivalence between well-posedness of the
\textit{Poisson problem}
$$
\begin{array}{lcl}
-\cA u&=&f \\
u_{|_{\partial\Omega}}&=& 0
\end{array} \leqno(P) 
$$
and the \textit{Dirichlet problem}
$$
\begin{array}{lcl}
\cA u&=&0 \\
u_{|_{\partial\Omega}}&=& g
\end{array} \leqno(D) 
$$
where $f\in L^n(\Omega)$ and $g\in C(\partial\Omega)$ are given. We consider
the operator $\cA$ defined in the previous section and define its realization
$A$ in $L^n(\Omega)$ (recall that $\Omega\subset \R^n$) by
\begin{eqnarray*}
D(A)&:=&\{u\in C_0(\Omega) \cap W^{2,n}_{\loc}(\Omega):\cA u \in L^n(\Omega)\}\\
Au&:=&\cA u \ .
\end{eqnarray*}

Thus the Poisson problem can be formulated in a more precise way by asking
under which conditions $A$ is invertible (i.e. bijective from $D(A)$ to
$L^n(\Omega)$ with bounded inverse $A^{-1}:L^n(\Omega) \to L^n(\Omega))$.
Note that for $\mu > 0$, the operator $A-\mu:=A-\mu I$ has the same form as $A$
(the order-0-coefficient $c$ being just replaced by $c-\mu$).

\begin{proposition}\label{prop2.1}
The operator $A$ is closed and injective. Thus, $A$ is invertible whenever it
is surjective. If $A-\mu$ is invertible for some $\mu\ge 0$, then it is so for
all.
\end{proposition}

\begin{proof}
By the Aleksandrov maximum principle (Theorem \ref{thmA1}) there exists a
constant $c_1>0$ such that 
\begin{equation}\label{equ2.1}
\|u\|_\infty \le 2 c_1 \|\mu u-Au\|_{L^n(\Omega)}
\end{equation}
for all $u \in D(A),\mu \ge 0$. In order to show that $A$ is closed, let $u_k
\in D(A)$ such that $u_k\to u$ in $L^n(\Omega)$ and $Au_k\to f$ in
$L^n(\Omega)$. 
It follows from (\ref{equ2.1}) that $u\in C_0(\Omega)$ and
$\lim\limits_{k\to \infty} u_k=u$ in
$C_0(\Omega)$. Let $B_{2\varrho}$ be a ball of radius $2\varrho$ such that
$\overline{B_{2\varrho}} \subset \Omega$. By the Calderon-Zygmund estimate
(Theorem \ref{thmA2})
$$
\|u_k\|_{W^{2,n}(B_\varrho)} \le c_\varrho (\|u_k\|_{L^n(B_{2\varrho})} +
\|Au_k\|_{L^n(B_{2\varrho})}) \ .
$$
It follows that $(u_k)_{k\in\N}$ is bounded in $W^{2,n}(B_\varrho)$. By passing
to a subsequence we can assume that $u_k\rightharpoonup u$ in
$W^{2,n}(B_\varrho)$. Consequently $\cA u_k\rightharpoonup \cA u$ in
$L^n(B_\varrho)$. Thus $\cA u=f$ on $B_\varrho$. Since the ball is arbitrary, it
follows that
$u\in D(A)$  and $Au=f$.

Now assume that $\mu_1-A$ is invertible for some $\mu_1 \ge 0$. Let $\mu_2 \ge
0$. Define $B(t)=t(\mu_1-A)+(1-t)(\mu_2 - A)$. Since $(\mu_1-A),(\mu_2-A)\in
\cL (D(A),L^n(\Omega))$ where $D(A)$ is considered as a Banach space with
respect to the graph norm $\|u\|_A:=\|u\|_{L^n(\Omega)}+\|Au\|_{L^n(\Omega)}$,
since by (\ref{equ2.1})
$$
2c_1\|B(t)u\|_{L^n(\Omega)} \ge \|u\|_{C(\bar{\Omega})} \ge
\frac{1}{|\Omega|^{1/n}} \|u\|_{L^n(\Omega)} \ ,
$$
for all $t\in[0,1]$ and since $B(1)$ is invertible, it follows from
\cite[Theorem 5.2]{gil.tru}
that $B(0)$ is also invertible.
\end{proof}

We call a function $u$ on $\Omega$ \textit{$\cA$-harmonic} if $u \in
W^{2,p}_{\loc}(\Omega)$ for some $p>1$ and $\cA u=0$. By \cite[Theorem
9.16]{gil.tru} each $\cA$-harmonic function $u$ is in
$\bigcap\limits_{q>1}
W^{2,q}_{\loc}(\Omega)$.
Given $g\in C(\partial \Omega)$, the \textit{Dirichlet  problem} consists in
finding an $\cA$-harmonic function $u\in C(\bar{\Omega})$  such that
$u_{|_{\partial\Omega}}=g$. We say that $\Omega$ is \textit{$\cA$-regular}
if for each $g \in C(\partial\Omega)$ there is a solution of the Dirichlet
problem. Uniqueness follows from the maximum principle \cite[Theorem
9.6]{gil.tru}
\begin{equation}\label{equ2.2}
-\|u^{-}_{|_{\partial\Omega}}\|_{L^{\infty}(\partial\Omega)}\le u(x)\le
\|u^+_{|_{\partial\Omega}}\|_{L^{\infty}(\partial\Omega)}
\end{equation}
for all $x \in \bar{\Omega}$, which holds for each $\cA$-harmonic function $u
\in C(\bar{\Omega})$. In particular,
\begin{equation}\label{equ2.3}
\|u\|_{C(\bar\Omega)} \le \|u\|_{C(\partial\Omega)} \ .
\end{equation}

\begin{theorem}\label{thm2.2}
The operator $A$ is invertible if and only  if $\Omega$ is $\cA$-regular.
\end{theorem}

\begin{proof}
a) Assume that $A$ is invertible.\\
\textit{First step:}
Let $g\in C(\partial \Omega)$ be of the form $g=G_{|_{\partial\Omega}}$
where $G\in C^2(\bar{\Omega})$. Then $\cA G \in L^n (\Omega)$. Let
$v=A^{-1}(\cA G)$, then $u:=G-v$ solves the
Dirichlet problem for $g$.\\
\textit{Second step:}
Let $g\in C(\partial\Omega)$ be arbitrary. Extending $g$ continuously and
mollifying we find $g_k\in C(\partial \Omega)$ of the kind considered in the
first step
such that $g=\lim\limits_{k\to \infty}g_k$ in $C(\partial \Omega)$. Let $u_k\in
C(\bar{\Omega})$ be $\cA$-harmonic satisfying $u_k{|_{\partial\Omega}}=g_k$.
By (\ref{equ2.3})
$u:=\lim\limits_{k\to \infty} u_k \ \mbox{ exists in } \ C(\bar{\Omega})$.
In particular, $u{|_{\partial\Omega}}=g$.
Let $\overline{B_{2\varrho}}\subset \Omega$. Then by the Calderon-Zygmund
estimate Theorem \ref{thmA2}
$$
\|u_k\|_{W^{2,p}(B_{\varrho})} \le c_\varrho \|u_k\|_{L^p(B_{2\varrho})} \le
c_\varrho c \|u_k\|_{C(\bar{\Omega})}
$$
(remember that $\cA u_k=0$). Thus $(u_k)_{k\in\N}$ is bounded in
$W^{2,p}(B_\varrho)$. Passing to a subsequence,
we can assume that $u_k\rightharpoonup u$ in $W^{2,p}(B_\varrho)$. This implies
that $\cA u=0$ in
$B_\varrho$. Since the
ball is arbitrary, it follows that $u$ is $\cA$-harmonic. Thus $u$ is a
solution of the Dirichlet problem $(D)$.\\
b) Conversely, assume that $\Omega$ is $\cA$-regular. Let $f\in L^n(\Omega)$.
We want to find $u \in D(A)$  such that $Au=f$. Let $B$ be a ball containig
$\bar{\Omega}$ and extend $f$ by $0$ to $B$. Then by Theorem \ref{thm1.4} we 
find $v\in C_0(B)\cap W^{2,n}_{\loc}(B)$ such that $\tilde{\cA v}=f$. 
Here $\tilde{\cA}$ is an extension of $\cA$ to the ball $B$ according to Lemma
\ref{lem1.3}a.
Let
$g=v_{|_{\partial\Omega}}$. Then by our assumption there exists an
$\cA$-harmonic function $w \in C(\bar{\Omega})$ such that
$w_{|_{\partial\Omega}}=g$. Let $u=v-w$. Then $u\in C_0(\Omega)\cap
W^{2,n}_{\loc}(\Omega)$ and $\cA u=\cA v=f$; i.e. $u \in D(A)$ and $Au=f$. We
have shown that $A$ is surjective, which implies invertibility by
Proposition \ref{prop2.1}.
\end{proof}

\begin{corollary}\label{cor2.3}
Assume that one of the following two conditions is satisfied:\\
a) $\Omega$ is Wiener regular and the coefficients $a_{ij}$ are globally
Lipschitz
continuous, or\\
b) $\Omega$ satisfies the exterior cone condition.\\
Then $\Omega$ is $\cA$-regular. More generally, for all $f\in L^n(\Omega),g\in
C(\partial \Omega)$ there exists a unique $u\in C(\bar{\Omega})\cap
W^{2,n}_{\loc}(\Omega)$ satisfying
\begin{eqnarray*}
-\cA u&=&f\\
u_{|_{\partial\Omega}}&=& g \ .
\end{eqnarray*}
\end{corollary}

\begin{proof}
Since $A$ is closed by Proposition \ref{prop2.1}
it follows from Theorem \ref{thm1.1} (in the case a)) and from Theorem
\ref{thm1.4}
(in the case b)) that $A$ is invertible. Thus $\Omega$ is $\cA$ regular by
Theorem \ref{thm2.2}. Let $f\in L^n(\Omega),g\in C(\partial \Omega)$. Since
$\Omega$ is $\cA$-regular, there exists an $\cA$-harmonic function $u_1 \in
C(\bar{\Omega})$ such that $u_{1_{|_{\partial\Omega}}}=g$. Since $A$ is
invertible, there exists a function $u_0 \in C_0(\Omega)\cap
W^{2,n}_{\loc}(\Omega)$ such that $-\cA u_0=f$. Let $u:=u_0+u_1$. Then $u\in
C(\bar{\Omega})\cap W^{2,n}_{\loc}(\Omega), u_{|_{\partial\Omega}}=g$ and $-\cA
u=f$. Uniqueness follows from Theorem \ref{thmA1}.
\end{proof}

For the Laplacian $\cA=\Delta,\Delta$-regularity is the usual regularity of
$\Omega$ with
respect to
the classical Dirichlet problem, which is frequently called
\textit{Wiener-regularity} because of Wiener's characterization via capacity
\cite[(2.37)]{gil.tru}. It
is a most interesting question how $\cA$-regularity and $\Delta$-regularity are
related. In general it is not true that $\cA$-regularity implies Wiener
regularity. In fact, K. Miller \cite{Mil} gives an example of an elliptic
operator $\cA$ with $b_j=c=0$ such that the pointed unit disc $\{x\in
\R^2:0<|x|<1\}$ is $\cA$-regular even though it is not $\Delta$-regular. The
other implication seems to be open. The fact that the uniform exterior cone
property (which is much stronger than $\Delta$-regularity) implies
$\cA$-regularity (Corollary \ref{cor2.3}) had been proved before by Krylov
\cite[Theorem 5]{Kry67} with  the help of probabilistic methods. If $\Omega$ is
merely $\Delta$-regular, then it seems not to be known whether $\Omega$ is
$\cA$-regular. Known results concerning this question are based on further
restrictive conditions on the coefficients $a_{ij}$. In Theorem \ref{thm1.1} we
gave a proof for globally Lipschitz continuous $a_{ij}$.
The best result seems to
be \cite[Theorem 4]{Kry67}  which goes in both directions:
If the $a_{ij}$ are Dini-continuous (in particular, if they are
H\"older-continuous), then $\Omega$ is $\Delta$-regular if and only if $\Omega$
is $\cA$-regular.

\section{Generation results}\label{secGener}
An operator $B$ on a complex Banach space $X$ is said to \textit{generate a
bounded holomorphic semigroup} if $(\lambda-B)$ is invertible for $\Re \lambda
> 0$ and
$$
\sup\limits_{\Re \lambda > 0} \|\lambda (\lambda-B)^{-1}\| < \infty \ .
$$
Then there exist $\theta \in (0,\pi/2)$ and a holomorphic bounded function
$T:\Sigma_\theta \to \cL(X)$ satisfying $T(z_1+z_2)=T(z_1)T(z_2)$ such that
\begin{equation}\label{equ3.1}
\lim\limits_{n\to \infty} e^{tB_n}=T(t) \ \mbox{ in } \ \cL(X) 
\end{equation}
for all $t>0$, where $B_n=nB(n-B)^{-1} \in \cL(X)$. Here $\Sigma_\theta$ is the
sector $\Sigma_\theta:=\{re^{i\alpha}:r>0,|\alpha|<\theta\} $.

If $B$ is an operator on a reel Banach space $X$ we say that $B$
\textit{generates a bounded holomorphic semigroup} if its linear extension
$B_{\C}$ to the complexification $X_{\C}$ of $X$ generates a bounded
holomorphic semigroup $T_\C$ on $X_\C$. In that case
$T_{\C}(t)X\subseteq X$ (see \cite[Corollary 2.1.3]{Lu}); in
particular $T(t):=T_{\C}(t)_{|X} \in \cL(X)$. We call $T=(T(t))_{t>0}$ the
semigroup generated by $B$. It satisfies $\lim\limits_{t\downarrow 0}T(t)x=x$
for all $x \in X$ (i.e., it is a $C_0$-semigroup) if and only if
$\overline{D(B)}=X$. We refer to \cite[Chapter 2]{Lu} and \cite[Sec.
3.7]{ABHN01} for these facts and further information.\\

In this section we consider the parts $A_c$ and $A_0$ of $\cA$ in
$C(\bar{\Omega})$ and $C_0(\Omega)$ as follows:
\begin{eqnarray*}
D(A_c)&:=&\{u\in C_0(\Omega)\cap W^{2,n}_{\loc}(\Omega):\cA u \in
C(\bar{\Omega})\}\\
A_c u&:=& \cA u \quad \mbox{ and}\\
D(A_0)&:=&\{u\in C_0(\Omega)\cap W^{2,n}_{\loc}(\Omega):\cA u \in
C_0(\Omega)\}\\
A_0u&:=& \cA u \ . 
\end{eqnarray*}

Thus $A_c$ is the part of $A$ in $C(\bar{\Omega})$ and $A_0$ the part of $A_c$
in $C_0(\Omega)$. Note that $D(A_0)\subseteq D(A_c)\subseteq
\bigcap\limits_{q>1}W^{2,q}_{\loc}$ by \cite[Lemma 9.16]{gil.tru}.
The main result of this section is the following.

\begin{theorem}\label{thm3.1}
Assume that $\Omega$ is $\cA$-regular. Then $A_c$ generates a bounded
holomorphic semigroup $T$ on $C(\bar{\Omega})$. The operator $A_0$ generates a
bounded holomorphic $C_0$-semigroup $T_0$ on $C_0(\Omega)$. Moreover,
$T(t)C_0(\Omega)\subseteq C_0(\Omega)$ and
$$
T_0(t)=T(t)_{|_{C_{0}(\Omega)}} \ .
$$
\end{theorem}

Recall that $\Omega$ is $\cA$-regular if one of the following conditions is
satisfied:
\begin{enumerate}
 \item [(a)]
$\Omega$ satisfies the uniform exterior cone condition or
 \item [(b)]
$\Omega$ is Wiener regular and the coefficients $a_{ij}$ are Dini-continuous.\\
In particular, $\Omega$ is $\cA$-regular if
 \item [(a')]
$\Omega$ is a Lipschitz-domain or
 \item [(b')]
$\Omega$ is Wiener-regular and the $a_{ij}$ are H\"older continuous.
\end{enumerate}

In the following complex maximum principle (Proposition \ref{prop3.3}) we
extend $\cA$ to the complex space $W^{2,p}_{\loc}(\Omega)$ without changing the
notation. We first proof a lemma.

\begin{lemma}\label{lem3.2}
Let $B\subseteq \Omega$ be a ball of center $x_0$ and let $u\in
W^{2,p}(B),p>n$, be a
complex-valued function such that $\cA u \in C(B)$. If $|u(x_0)|\ge |u(x)|$ for
all $x \in B$, then
$$
\Re \left[ \overline{u(x_0)} (\cA u)(x_0\right] \le 0 \ .
$$
\end{lemma}

\begin{proof}
We  may assume that $x_0=0$. If the claim is wrong, then there exist
$\varepsilon > 0$ and a ball $B_\varrho \subset B$ such that $\Re \left[
\overline{u(x)} (\cA u)(x)\right] \ge \varepsilon$ on $B_\varrho$.\\
Since $\partial_j|u|^2=(\partial_j u){\bar{u}}+u\overline{\partial_ju}=2 \Re
\left[ \partial_j u \bar{u}\right]$, and
$\partial_{ij}(u\bar{u})=(\partial_{ij}u)\bar{u}+\partial_iu\overline{\partial_j
u} +\partial_ju\overline{\partial_i u}+u\overline{\partial_{ij}u}$, and since by
ellipticity
$$
\Re \sum\limits_{i,j}a_{ij}\partial_i u\overline{\partial_ju}\ge 0 \ , \ \Re
\sum\limits_{i,j}a_{ij}\partial_j u \overline{\partial_iu} \ge 0 \ ,
$$
it follows that
\begin{eqnarray*}
\cA |u|^2&\ge & \Re \sum\limits_{i,j}a_{ij} (\partial_{ij} u)\bar{u}
+\Re \sum\limits_{i,j}a_{ij}u\overline{\partial_{ij}u}\\
&& + \sum\limits_j b_j 2 \Re \left[\partial_ju\bar{u} \right] +cu\bar{u}\\
&\ge& 2\Re \left( \cA u\bar{u}\right) \ge 2 \varepsilon \ \mbox{ on } \
B_\varrho \ .
\end{eqnarray*}
Let $\psi(x)=|u|^2-\tau|x|^2,\tau > 0$. Then $\cA|\psi|^2\ge 2 \varepsilon
-c_1\tau$ on $B_\varrho$ for all $\tau > 0$ and some $c_1>0$. Choosing $\tau >0$
small enough, we have $\cA|\psi|^2 \ge \varepsilon$ on $B_\varrho$.\\
Since $\psi \in W^{2,p}(B_\varrho)\cap C(\overline{B_\varrho})$, by
Aleksandrov's  maximum principle \cite[Theorem 9.1]{gil.tru}, see Theorem \ref
{thmA1}, it follows that
\begin{eqnarray*}
|u(0)|^2=|\psi(0)|^2&\le&\sup\limits_{\partial B_\varrho(0)} \psi\\
&=&\sup\limits_{\partial B_\varrho(0)} |u|^2-\tau\varrho^2\\
&\le& |u(0)|^2-\tau\varrho^2 < |u(0)|^2 \ ,
\end{eqnarray*}
a contradiction.
\end{proof}

\begin{proposition}\label{prop3.3} \textit{(complex maximum principle)}.
Let $u\in C(\bar{\Omega})\cap W^{2,n}_{\loc}(\Omega)$ such that $\lambda u-\cA
u=0$ where $\Re \lambda > 0$. If there exists $x_0 \in \Omega$  such that
$|u(x)|\le |u(x_0)|$   
for all $x\in \Omega$, then $u\equiv 0$. Consequently,
$$
\max\limits_{\bar{\Omega}} |u(x)|=\max\limits_{\partial \Omega} |u(x)| \ .
$$
\end{proposition}

\begin{proof}
If $|u(x)| \le |u(x_0)|$ for all $x \in \Omega$, then by Lemma \ref{lem3.2},
$\Re \left[\overline{u(x_0)}(\cA u)(x_0) \right] \le 0$. Since $\lambda u = \cA
u$, it follows that
$$
\Re \lambda |u(x_0)|^2 = \Re \left[\overline{u(x_0)}(\cA u)(x_0) \right] \le 0
\ .
$$
Hence $u(x_0)=0$.
\end{proof}

Next, recall that an operator $B$ on a real Banach space $X$ is called
$m$\textit{-dissipative} if $\lambda-B$ is invertible and
$$
\lambda \|(\lambda-B)^{-1}\|\le 1 \ \mbox{ for all } \ \lambda > 0 \ .
$$

Now we show that the operator $A_c$ is $m$-dissipative and that the resolvent
is positive (i.e., maps non-negative functions to non-negative functions).

\begin{proposition}\label{prop3.5}
Assume that $\Omega$ is $\cA$-regular. Then $A_c$ is $m$-dissipative and
$(\lambda-A_c)^{-1} \ge 0$ for $\lambda > 0$.
\end{proposition}

\begin{proof}
Let $\lambda > 0$. Since by Theorem \ref{thm2.2} the operator $(\lambda-A)$ is
bijective, also $(\lambda-A_c)$ is bijective.\\
a) We show that $(\lambda-A_c)^{-1} \ge 0$.
Let $f\in C(\bar{\Omega}),f \le 0, u:=(\lambda-A_c)^{-1}f$.
Assume that $u^+ \neq 0$. Since $u\in C_0(\Omega)$, there exists $x_0 \in
\Omega$  such that $u(x_0)=\max\limits_\Omega u > 0$. Then by Lemma
\ref{lem3.2}, $\cA u(x_0)\le 0$. Since $\lambda u-\cA u =f$, it follows that
$\lambda u(x_0)\le f(x_0) \le 0$ a contradiction.\\
b) Let $f\in C(\bar{\Omega}),u=(\lambda-A_c)^{-1}f$. We show that $\|\lambda
u\|_{C(\bar{\Omega})} \le \|f\|_{C(\bar{\Omega})}$.
Assume  first that $f \ge 0, f \neq 0$. Then $u\ge 0$ by a) and $u\neq 0$. Let
$x_0 \in \Omega$ such that $u(x_0)=\|u\|_{C(\bar{\Omega})}$.
Then $(A_c u)(x_0) \le 0$ by Lemma \ref{lem3.2}. Hence $\lambda u(x_0) \le
\lambda u(x_0)-(A_c u)(x_0) = f(x_0) \le \|f\|_{C(\bar{\Omega})}$.\\
If $f\in C(\bar{\Omega})$ is  arbitrary, then by a) $|(\lambda-A_c)^{-1}f| \le
(\lambda-A_c)^{-1}|f|$ and so $\|\lambda
(\lambda-A_c)^{-1}f\|_{C(\bar{\Omega})} \le \|f\|_{C(\bar{\Omega})}$.
\end{proof}

Now we consider the complex extension of $A_c$ (still denoted by $A_c$) to the
space of all complex-valued functions on $\bar{\Omega}$ which we still denote
by $C(\bar{\Omega})$. Our aim is to prove that for $\Re \lambda > 0$ the
operator $(\lambda-A_c)^{-1}$ is invertible and
$$
\|(\lambda-A_c)^{-1}\| \le \frac{M}{|\lambda |}  \ ,
$$
where $M$ is a constant.
For that, we extend the coefficents $a_{ij}$ to uniformly continuous bounded
real-valued functions on $\R^n$ satisfying the strict ellipticity condition
$$
\Re\sum\limits^d_{i,j=1}a_{ij}(x)\xi_i\bar{\xi}_j \ge \frac{\Lambda}{2}|\xi|^2
$$
$(\xi \in \R^n,x\in \R^n)$, keeping the some notation, see Lemma \ref{lem1.3}a.
We extend $b_j,c$ to bounded measurable functions
on $\R^n$ such that $c\le 0$ (keeping the same notation).
Now we define the operator $B_\infty$ on $L^\infty(\R^n)$ by
$$
\begin{array}{llcl}
&D(B_\infty)&:=&\{u\in \bigcap\limits_{p>1} W^{2,p}_{\loc}(\R^n):u,\cB u \in
L^\infty(\R^n)\}\\
\mbox{where }\qquad & B_\infty u&:=&\cB u \ , \\
& B_\infty u&:=&
\sum\limits^d_{i,j=1}a_{ij}\partial_{ij}u+\sum\limits^d_{j=1}b_j\partial_j u+cu
\mbox{ for } u \in W^{2,p}_{\loc}(\R^n) \ .
\end{array}
$$
The operator $B_\infty$ is   sectorial. This is proved in \cite[Theorem
3.1.7]{Lu} under the assumption that the coefficients $b_j,c$ are  uniformly
continuous. We give a perturbation argument to deduce the general case from the
case $b_j=c=0$. The following lemma shows in particular that the domain of
$B_\infty$ is independent of $b_j$ and $c$.

\begin{lemma}\label{lem3.6}
One has $D(B_\infty)\subset W^{1,\infty}(\R^n)$.  Moreover, for each
$\varepsilon > 0$ there exists $c_\varepsilon \ge 0$ such that
$$
\|u\|_{W^{1,\infty}(\R^n)} \le \varepsilon \|B_\infty
u\|_{L^\infty(\R^n)}+c_\varepsilon \|u\|_{L^\infty(\R^n)}
$$
for all $u\in D(B_\infty)$.
\end{lemma}

\begin{proof}
Consider an arbitrary ball $B_1$ in $\R^n$ of radius 1 and the corresponding
ball $B_2$ of radius 2. Let $p>n$. Since the injection of $W^{2,p}(B_1)$ into
$C^1(\bar{B_1})$ is compact, for each $\varepsilon > 0$ there exists
$c^{\prime}_\varepsilon > 0$ such that
$$
\|u\|_{C^1(\bar{B}_1)} \le \varepsilon
\|u\|_{W^{2,p}(B_1)}+c^{\prime}_\varepsilon
\|u\|_{L^\infty(B_1)} \ .
$$
By the Calderon-Zygmund estimate this implies that
\begin{eqnarray*}
\|u\|_{C^1(\bar{B}_1)}&\le& \varepsilon c_1 (\|B_\infty
u\|_{L^\infty(B_2)}+\|u\|_{L^\infty(B_2)})\\
&&+ c^\prime_\varepsilon \|u\|_{L^\infty(B_1)}\\
&\le&\varepsilon c_1\|B_\infty u\|_{L^\infty(\R^n)}+(\varepsilon
c_1+c^\prime_\varepsilon) \cdot \|u\|_{L^\infty(\R^n)} \ .
\end{eqnarray*}
Since $\|u\|_{L^\infty(\R^n)} = \sup\limits_{B_1} \|u\|_{L^\infty(B_1)}$, where
the supremum is taken over all balls of radius $1$ in $\R^n$, the claim
follows.
\end{proof}

\begin{theorem}\label{thm3.7}
There exist $M\ge 0, \omega \in \R$   such that $(\lambda-B_\infty)$ is
invertible and
$$
\|\lambda(\lambda-B_\infty)^{-1}\| \le M \quad (\Re \lambda > \omega) \ .
$$
\end{theorem}

\begin{proof}
Denote by $B^0_\infty$ the operator with the coefficients $b_j,c$ replaced by
$0$. Lemma \ref{lem3.6} implies that $D(B^0_\infty)=D(B_\infty)$ and (applied
to $B^0_\infty)$ that
$$
\|(B_\infty-B^0_\infty)u\|_{L^\infty(\R^n)}\le \varepsilon \|B^0_\infty
u\|_{L^\infty(\R^n)} +c^\prime_\varepsilon \|u\|_{L^\infty(\R^n)}
$$
for all $u\in D(B_\infty^0),\varepsilon > 0$ and  some $c^\prime_\varepsilon \ge
0$.
Since $B^0_\infty$ is sectorial by \cite[Theorem 3.1.7]{Lu} the claim
follows from the usual holomorphic perturbation result \cite[Theorem
3.7.23]{ABHN01}.
\end{proof}

Now we use the maximum principle, Lemma \ref{lem3.2}, to carry over the
sectorial estimate from $\R^n$ to $\Omega$. This is done in
a very abstract framework by Lumer-Paquet \cite{LP}, see \cite[Section 2.5]{Ar}
for the
Laplacian.\\

\noindent
\textbf{Proof of  Theorem \ref{thm3.1}.}
Let $\omega$ be the constant from Theorem \ref{thm3.7} and let $\Re \lambda >
\omega, f \in C(\bar{\Omega}),u=(\lambda-A_c)^{-1}f$. Then
$$
u\in C_0(\Omega)\cap \bigcap\limits_{p>1} W^{2,p}_{\loc} (\Omega) \ \mbox{ and }
\ \lambda u -\cA u = f \ .
$$
Extend $f$ by $0$ to $\R^n$ and let $v=(\lambda-B_\infty)^{-1} f$.
Then $\lambda v-\cA v=f$ on $\Omega$ and $\|\lambda v\|_{L^\infty(\Omega)}\le
M\|f\|_{C(\bar{\Omega})}$ by Theorem \ref{thm3.7}. Moreover, $w:=v-u \in
C(\bar{\Omega})\cap
\bigcap\limits_{p\ge 1} W^{2,p}_{\loc}(\Omega),\lambda w -\cA w = 0$
on $\Omega$ and $w(z)=v(z)$ for all $z \in \partial \Omega$. Then by the
complex maximum principle Proposition \ref{prop3.3},
$$
\|w\|_{C(\bar{\Omega})}=\max\limits_{z\in\partial \Omega} |v(z)| \le
\frac{M}{|\lambda |} \|f\|_{C(\bar{\Omega})}  \ .
$$
Consequently,
\begin{eqnarray*}
\|u\|_{C(\bar{\Omega})}&=& \|u-v+v\|_{C(\bar{\Omega})}\\
&\le&\|w\|_{C(\bar{\Omega})} + \|v\|_{C(\bar{\Omega})}\\
&\le& \frac{2M}{|\lambda |}\|f\|_{C(\bar{\Omega})} \ .
\end{eqnarray*}
This is the desired estimate which shows that $A_c$ is sectorial.
By \cite[Proposition 2.1.11]{Lu} there exist a sector $\Sigma_\theta +
\omega :=\{\omega + re^{i\alpha}:r>0,|\alpha| < \theta\}$ with $\theta \in
(\frac{\pi}{2},\pi)$, $\omega \ge 0$, and a constant $M_1 > 0$ such that
$$
(\lambda-A_c)^{-1} \ \mbox{ exists for } \ \lambda \in \Sigma_\theta + \omega \
\mbox{ and } \ \|\lambda(\lambda- A_c)^{-1}\| \le M_1 \ .
$$
Thus there exists $r>0$ such that $(\lambda-A_c)$ is invertible and
$\|\lambda(\lambda-A_c)^{-1}\|\le M$ whenever $\Re \lambda > 0$ and
$|\lambda|>r$.
Since $A$ is invertible by Theorem \ref{thm2.2}, it follows that $A_c$ is
bijective. Since the resolvent set of $A_c$ is nonempty, $A_c$ is  closed. Thus
$A_c$ is invertible. Since by Proposition \ref{prop3.5} $A_c$ is resolvent
positive, it follows from \cite[Proposition 3.11.2]{ABHN01} that there exists
$\varepsilon > 0$ such that $(\lambda-A_c)$ is invertible whenever $\Re \lambda
> - \varepsilon$. As a consequence,
$$
\sup\limits_{|\lambda|\le r \atop \Re \lambda > 0}
\|\lambda(\lambda-A_c)^{-1}\| < \infty \ .
$$
Together with the previous estimates this implies that
$$
\| \lambda (\lambda-A_c)^{-1}\| \le M_2
$$
whenever $\Re \lambda > 0$ for some constant $M_2$. Thus $A_c$ generates a
bounded
holomorphic semigroup $T$ on $C(\bar{\Omega})$.
Since $D(A_c)\subset C_0(\Omega)$ and $\cD(\Omega) \subset D(A_c)$ it follows
that $\overline{D(A_c)}=C_0(\Omega)$. The  part of $A_c$ in $C_0(\Omega)$ is
$A_0$. So it follows from \cite[Remark 2.1.5, Proposition 2.1.4]{Lu} that
$A_0$ generates a bounded, holomorphic $C_0$-semigroup $T_0$ on $C_0(\Omega)$ 
and $T_0(t)=T(t)_{|_{C_0(\Omega)}}$ on $C_0(\Omega)$.
\qed\\

Finally we  mention compactness and strict positivity.

\begin{proposition}\label{prop3.8}
Assume that $\Omega$ satisfies the uniform exterior cone condition. Then
$(\lambda-A_c)^{-1}$ and $T(t)$ are compact operators $(\lambda>0,t>0)$.
\end{proposition}

\begin{proof}
It follows from Theorem \ref{thmA3} that $D(A_c)\subset C^\alpha(\Omega)$. Since
the
embedding of $C^\alpha(\Omega)$ into $C(\bar{\Omega})$ is compact, it follows
that the resolvent of $A_c$ is compact. Since $T$ is holomorphic, it follows
that $T(t)$ is compact for all $t>0$.
\end{proof}

\begin{proposition}\label{prop3.9}
Assume that $\Omega$ is $\cA$-regular. Let $t>0,0\le f \in C_0(\Omega),f
\not\equiv 0$. Then $(T_0(t)f)(x)>0$ for all $x \in \Omega$.
\end{proposition}

\begin{proof}
a) We show that $u:=(\lambda-A_0)^{-1}f$ is strictly positive. Assume that
$u(x)\le 0$ for some $x\in \Omega$. Let $v=-u$. Then $\cA v-\lambda v=f \ge 0$.
It follows from the maximum principle \cite[Theorem 9.6]{gil.tru} that $v$ is 
constant. Since $v \in C_0(\Omega)$, it follows that $v\equiv 0$. Hence also
$f\equiv 0$.\\
b) It follows from a) that $T_0$ is a positive, irreducible $C_0$-semigroup on
$C_0(\Omega)$. Since the semigroup is holomorphic, the claim follows from
\cite[C-III.Theorem 3.2.(b)]{Na}.
\end{proof}

\begin{appendix}

\section{Results on elliptic partial differential equations}
In this section, we collect some results on elliptic partial differential
equations, which can be found in text books, for example \cite{gil.tru}. We
consider the elliptic operator $\cA$ from the Introduction and assume that the
ellipticity constant $\Lambda >0$ is so small that $\|a_{ij}\|_{L^\infty} \ , \
\|b_j\|_{L^\infty} \ ,   \ \|c\|_{L^\infty} \le \frac{1}{\Lambda}$.\\

\begin{theorem}[Aleksandrov's  maximum principle, \mbox{\cite[Theorem
9.1]{gil.tru}}] \label{thmA1}
Let $f \in L^n(\Omega),u \in C(\bar{\Omega})\cap W^{2,n}_{\loc}(\Omega)$ such
that
$$
-\cA u \le f \ .
$$
Then
$$
\sup\limits_{\Omega} u \le \sup\limits_{\partial \Omega} u^+ +
c_1\|f^+\|_{L^n(\Omega)}
$$
where the constant $c_1$ depends merely on $n,\diam \Omega$ and
$\|b_j\|_{L^n(\Omega)},j=1\ldots,n$.
Consequently, if $u\in C_0(\Omega)$ and $-\cA u =f$, then 
$$
\|u\|_{L^\infty(\Omega)} \le 2c_1\|f\|_{L^n(\Omega))} 
$$
and $u \le 0$ if $f \le 0$.
\end{theorem}

\begin{theorem} [Interior Calderon-Zygmund estimate, \mbox{\cite[Theorem
9.11]{gil.tru}}] \label{thmA2}
Let $B_{2\varrho}$ be a ball of radius $2\varrho$ such that
$\overline{B_{2\varrho}} \subset \Omega$, and let $u\in W^{2,p}(B_{2\varrho})$,
where $1<p
<\infty$. Then
$$
\|u\|_{W^{2,p}(B_{\varrho})} \le c_\varrho (\|\cA
u\|_{L^p(B_{2\varrho})}+\|u\|_{L^p(B_{2\varrho})}) 
$$
where $B_\varrho$ is the ball of radius $\varrho$ concentric with
$B_{2\varrho}$. The constant $c$ merely depends on
$\Lambda,n,\varrho,p$ and the continuity moduli of the $a_{ij}$.
\end{theorem}

\begin{theorem} [H\"older regularity, \mbox{\cite[Corollary 9.29]{gil.tru}}]
\label{thmA3}
Assume that $\Omega$ satisfies the uniform exterior cone condition. Let
$u \in C_0(\Omega)\cap W^{2,n}_{\loc}(\Omega)$ and $f \in L^n(\Omega)$ such
that $-\cA u=f$. Then $u \in C^\alpha(\Omega)$ and
$$
\|u\|_{C^\alpha(\Omega)} \le C(\|f\|_{L^n(\Omega)}+\|u\|_{L^2(\Omega)})
$$
where $\alpha > 0$ and $c >0$ depend merely on $\Omega,\Lambda$ and $n$. 
\end{theorem}

In \cite[Corollary 9.29]{gil.tru} it is supposed that $u \in W^{2,n}(\Omega)$.
But an inspection of the proof and of the results preceding \cite[Corollary
9.29]{gil.tru} shows that $u\in W^{2,n}_{\loc}(\Omega)$ suffices.
The above H\"older regularity also holds for solutions of equations in
divergence form when the right-hand side $f$ is in $L^q(\Omega)$ for some
$q>\frac{n}{2}$, see \cite[Theorem 8.29]{gil.tru}.
\end{appendix}

\end{document}